\documentclass[11pt,english]{smfart} 
\usepackage[T1]{fontenc}
\usepackage[english,french]{babel}
\usepackage{amscd}
\usepackage{url,xspace,smfthm}
\usepackage{amssymb,latexsym}
\usepackage{amsbsy,bm} 
\usepackage{paralist}
\usepackage{datetime} 
\usepackage{colortbl}
\usepackage{color}
\usepackage[dvipsnames]{xcolor}
\usepackage{mathrsfs}  
\usepackage[colorlinks,	colorlinks,
linkcolor=BrickRed,
citecolor=Green,
urlcolor=Cerulean,hypertexnames=true]{hyperref} 
\usepackage{mathtools}
\usepackage{euscript}
\usepackage[all]{xy}
\usepackage{hyperref}
\usepackage{graphicx} 
\usepackage[text={150mm,251mm},centering, marginparwidth=75pt]{geometry} 
\usepackage{comment}  
\usepackage{cite}  
\usepackage{thmtools}
\usepackage{enumitem}
\usepackage{letltxmacro} 
\usepackage{nameref}
\usepackage{cleveref}
\usepackage{calc}
\usepackage{interval}
\usepackage{manfnt}
\usepackage{tikz-cd}  
\usepackage[utf8]{inputenc}
\usepackage{marginnote}
\usepackage{smfhyperref}
\usepackage[hyperpageref]{backref}
\renewcommand{\backref}[1]{$\uparrow$~#1}
\renewcommand\thefootnote{\textcolor{MidnightBlue}{\arabic{footnote}}}

\usepackage{faktor} 
\usepackage{scalerel}

\makeatletter     %Replace the Section ...  by the symbol \S ...  in Cref
\newcommand{\crefnames}[3]{%
	\@for\next\coloneqq#1\do{%
		\expandafter\crefname\expandafter{\next}{#2}{#3}%
	}%
}
\makeatother

\setlist[itemize]{wide = 0pt, labelwidth = 2em, labelsep*=0em, itemindent = 0pt, leftmargin = \dimexpr\labelwidth + \labelsep\relax, noitemsep,topsep = 1ex,}
\setlist[enumerate]{wide = 0pt, labelwidth = 2em, labelsep*=0em, itemindent = 0pt, leftmargin = \dimexpr\labelwidth + \labelsep\relax, noitemsep,topsep = 1ex}

\theoremstyle{plain}

\renewcommand{\thethmx}{\Alph{thmx}} 
\newtheorem{theorem}{Theorem}[section]

\newtheorem{proposition}[theorem]{Proposition}
\newtheorem{corollary}[theorem]{Corollary}

\theoremstyle{definition}
\newtheorem{definition}[theorem]{Definition} 
\theoremstyle{remark}
 
\newtheorem*{remark*}{Remark}

\numberwithin{equation}{section}  
\theoremstyle{plain}
\newlist{thmlist}{enumerate}{1}
\setlist[thmlist]{wide = 0pt, labelwidth = 2em, labelsep*=0em, itemindent = 0pt, leftmargin = \dimexpr\labelwidth + \labelsep\relax, noitemsep,topsep = 1ex, font=\normalfont, label=(\roman*), ref=\thetheorem.(\roman{thmlisti})}

\addtotheorempostheadhook[theorem]{\crefalias{thmlisti}{thm}}

\addtotheorempostheadhook[assumpsion]{\crefalias{thmlisti}{assumption}}

\addtotheorempostheadhook[corollary]{\crefalias{thmlisti}{cor}}

\addtotheorempostheadhook[proposition]{\crefalias{thmlisti}{proposition}}

\addtotheorempostheadhook[definition]{\crefalias{thmlisti}{dfn}}

\addtotheorempostheadhook[lemma]{\crefalias{thmlisti}{lem}}
\addtotheorempostheadhook[main]{\crefalias{thmlisti}{main}}

\addtotheorempostheadhook[remark]{\crefalias{thmlisti}{rem}}

\newlist{thmenum}{enumerate}{1} % also creates a counter called 'propenumi'
\setlist[thmenum]{wide = 0pt, labelwidth = 2em, labelsep*=0em, itemindent = 0pt, leftmargin = \dimexpr\labelwidth + \labelsep\relax, noitemsep,topsep = 1ex, font=\normalfont, label=(\roman*), ref=\thethmx.(\roman{thmenumi})}%{label=\alph*), ref=\thethmx~(\alph*)}
\crefalias{thmenumi}{thmx} 

\newlist{corlist}{enumerate}{1} % also creates a counter called 'propenumi'
\setlist[corlist]{wide = 0pt, labelwidth = 2em, labelsep*=0em, itemindent = 0pt, leftmargin = \dimexpr\labelwidth + \labelsep\relax, noitemsep,topsep = 1ex, font=\normalfont, label=(\roman*), ref=\thecorx.(\roman{corlisti})}%{label=\alph*), ref=\thethmx~(\alph*)}
\crefalias{corlisti}{corx} 

\crefname{lemma}{Lemma}{Lemmas} 
\crefname{conjecture}{Conjecture}{Conjectures}
\crefname{theorem}{Theorem}{Theorems}
\crefname{proposition}{Proposition}{Propositions}
\crefname{definition}{Definition}{Definitions}
\crefname{remark}{Remark}{Remarks}
\crefname{corollary}{Corollary}{Corollaries}
\crefname{corx}{Corollary}{Corollaries}
\crefname{problem}{Problem}{Problems}
\crefname{thmx}{Theorem}{Theorems}
\crefname{claim}{Claim}{Claims}
\crefname{assumption}{Assumption}{Assumptions}
\crefname{main}{Main Theorem}{Main Theorems}

\newcommand{\R}{\mathbb{R}}

\newcommand{\Z}{\mathbb{Z}}

\begin{document}

\title{The applications of Bieri-Neumann-Strebel invariant on K\"ahler groups}

{
\author{Yuan Liu}
\email{lexliu@hku.hk}
\address{The University of Hong Kong, Pokfulam Road, Hong Kong} 
\urladdr{https://sites.google.com/view/yuan-lius-website}
}
\date{\today}

\begin{abstract}
We give several applications of the Bieri-Neumann-Strebel invariant on K\"ahler groups.
Specifically, we provide simpler proof of the Napier-Ramachandran theorem on the HNN decomposition about K\"ahler groups and show that amenable K\"ahler groups have an empty complement of the BNS invariant.
\end{abstract}

\maketitle

\begingroup
\renewcommand{\thefootnote}{}
\footnotetext{\noindent%
    \phantomsection\makebox[0pt][l]{\textit{2020 Mathematics Subject Classification}: Primary 20F65, 32Q15.}}
%\\
%\makebox[0pt][l]
%{\textit{Keywords:} K\"ahler group, BNS invariant, HNN decomposition.}}
\addtocounter{footnote}{-1} % Prevents affecting subsequent footnotes
\endgroup

\section{Introduction}
We say a finitely presented group $G$ is a K\"ahler group if it can be realized as the fundamental group of a compact K\"ahler manifold. Not every finitely presented group is a K\"ahler group, and it is a question of Serre to characterize K\"ahler groups. The readers may refer to the monographs \cite{ABCKT} and \cite{PyBook} for this topic. The overall picture for the K\"ahler groups is still very unclear.

For $G$ a finitely generated group, Bieri, Neumann, and Strebel introduced a celebrated invariant $\Sigma(G)$ in \cite{BNS87} (hereafter, we call it the BNS invariant). Unfortunately, the computation of the BNS invariant is extremely difficult and is only known for restricted types of groups up to now. Surprisingly, the combination of these two mysterious objects, i.e.
(the complement of) the BNS invariant of a K\"ahler group, is described completely by Delzant in \cite{Delzant10}, using Simpson's Lefschetz theorem\cite{Simpson}.

It is thus quite natural to use this BNS invariant to put restrictions on K\"ahler groups. For the HNN decomposition, some results were known for a long time thanks to the work of Napier and Ramachandran \cite[Theorem 0.2 and Theorem 0.3]{NR2008}. Now, with the help of the BNS invariant, we give a much easier explanation (\cref{thm: nr 1,thm: nr 2}). Moreover, we show that amenable K\"ahler groups have the empty complement of the BNS invariant (\cref{prop:amenable_kahler}).

%Furthermore, the BNS invariant is applied to prove the Shafarevich conjecture. Recall that this conjecture asks if the universal covering of a compact K\"ahler manifold is holomorphically convex. This was initialed by the work of Koll\'ar\cite{Kollar,Kollar_book} and Campana \cite{Campana} independently and intensely studied ever since by many others (see, for example, \cite{Katzarkov,KR,Eys,Eys11,EKPR,CCE,DYK,DY2024}). Most importantly, the conjecture holds if the K\"ahler group is linear. These results mentioned above are all based on very complicated Hodge theory. Here, we provide a direct argument of this conjecture if the K\"ahler group is commeasurable to the fundamental group of a hyperbolic Riemann orbifold with positive genus. This is only a special case of the above-mentioned linear case, but we provide a much simpler proof with the help of the BNS invariant.

\bigskip

This note is organized as follows. In \cref{sec:bns}, we provide essential background on the BNS invariant. In \cref{sec:hnn}, we prove the results on HNN decompositions about K\"ahler groups and on amenable K\"ahler groups. 
%In \cref{sec:shafa}, we prove the Shafarevich conjecture if the K\"ahler group is commeasurable to the fundamental group of a hyperbolic Riemann orbifold with a positive genus.

\section{The Bieri-Neumann-Strebel invariant}\label{sec:bns}

\subsection{General theory}
In 1987, a famous group theoretic invariant was introduced by Bieri, Neumann, and Strebel \cite{BNS87}, which is now called the Bieri-Neumann-Strebel (BNS for short) invariant. For a finitely generated group $G$, this invariant gives the full information when the kernel of an abelian quotient of $G$ is finitely generated (for its precise meaning, see \cref{prop:bns_abelian_quotient}). The readers may refer to \cite{BNS87} or the unpublished monograph \cite{BSunpublished} for more details. \textit{From now on, we always assume that the first Betti number $b_1(G)>0$.}

Let $G$ be a finitely generated group and $\chi\colon G\to \R$ be a non-trivial group homomorphism. Denote $G_{\chi}=\{g\in G\colon \chi(g)\geqslant 0\}$. For any given set of generators $S$, denote the associated Cayley graph of $G$ as $C\coloneqq\mathrm{Cay}(G; S)$. We define its subgraph $C_{\chi}$ as follows: (i) a vertex in $C$ belongs to $C_{\chi}$ if its $\chi$-value is nonnegative, (ii) an edge in $C$ belongs to $C_{\chi}$ if \textit{both} vertices of this edge have nonnegative $\chi$-values. We can consider if $C_{\chi}$ is connected and record the $\chi$ when this happens. Notice that this property does not depend on the generating set $S$ (see \cite[Theorem 2.1]{BSunpublished} or \cite[Proposition 11.1]{PyBook}) and is invariant if we multiply $\chi$ by a positive real constant. Denote $[\chi]$ to be this equivalent relation of multiplication by a positive real constant, and write
$\mathrm{S}(G)=(\mathrm{Hom}(G;\R)-\{0\})\slash \R^{+}$ for the collection of all these equivalent classes.

\begin{definition}
With the same notation as above, for $G$ being a finitely generated group, we define its BNS invariant as
$$\Sigma(G)=\{[\chi]\in \mathrm{S}(G)\colon C_{\chi} \text{\ is connected}\}$$
\end{definition}

We need the following properties of this invariant for later use.

\begin{proposition}\label{prop:property_BNS_union}(\cite[Theorem C]{BNS87})
Let $G$ be a finitely presented group with no nonabelian free subgroups, then 
$$\Sigma(G)\cup -\Sigma(G)=\mathrm{S}(G)$$
where $-\Sigma(G)$ is the image of $\Sigma(G)$ under the antipodal map.
\end{proposition}

\begin{proposition}\label{prop:bns_abelian_quotient}(\cite[Theorem B1]{BNS87})
If $G$ is a finitely generated group and $A$ is its non-trivial abelian quotient with kernel $N$, then $N$ is finitely generated if and only if 
$$\mathrm{S}(G,N)\subseteq \Sigma(G)$$
where $\mathrm{S}(G,N)=\{[\chi]\in \mathrm{S}(G)\colon \chi(N)=0\}$. In particular, $G'$ is finitely generated if and only if $\Sigma(G)=\mathrm{S}(G)$, where $G'=[G,G]$ is the commutator subgroup of $G$.
\end{proposition}

%We have the following corollary (see \cite[Page 33, Corollary 5.2]{BSunpublished}), and we include its proof here for completeness.
%\begin{corollary}\label{coro:b_1>=2 fiber}
%Let $G$ be a finitely presented group with no nonabelian-free subgroups and $b_1(G)\geqslant 2$, then $G$ algebraically fibers, i.e. there exists nonzero $\chi\in\mathrm{Hom}(G;\Z)$ with $\mathrm{Ker}(\chi)$ finitely generated,
%\end{corollary}
%\begin{proof}
%We have that $\Sigma(G)\cup -\Sigma(G)=\mathrm{S}(G)$ by \cref{prop:property_BNS_union}.
%Since $\Sigma(G)$ is open, and $\mathrm{S}(G)$ is connected when $b_1(G)>1$, we have that $\Sigma(G)\cap -\Sigma(G)\neq \emptyset$. Since the set of rational points is dense in $\Sigma(G)$, we have nonzero $\chi\in \mathrm{Hom}(G;\Z)$ such that $\pm[\chi]\in \Sigma(G)$. By \cref{prop:bns_abelian_quotient}, we have $\mathrm{Ker}(\chi)$ is finitely generated.
%\end{proof}

\subsection{The BNS invariant for K\"ahler groups}
When this group $G$ is the fundamental group of a finite CW-complex $X$, we can give a topological explanation of this invariant \cite[section 5]{BNS87}. If we denote $\hat{X}$ as the maximal free abelian covering of $X$ with Galois group $H$. Notice that $\chi\in \mathrm{Hom}(G;\R)$ naturally factor through the maximal free abelian quotient $H$ of $G$, for which we still denote as $\chi: H\to \R$. Let $f:\hat{X}\to \R$ be the primitive function in the sense that $f(h.x)=f(x)+\chi(h)$ for any $h\in H$ and $x\in \hat{X}$. Now consider $\{x\in X\mid f(x)\geqslant 0\}$, it has a unique component on which $f$ is unbounded, and we denote it as $\hat{X}_{f}$ (\cite[Lemma 5.2]{BNS87}). Then we have $[\chi]\in \Sigma(G)$ if and only if the morphism on fundamental groups induced by the inclusion $\hat{X}_{f}\hookrightarrow \hat{X}$ is onto (\cite[Theroem 5.1]{BNS87}). Using this topological explanation along with Simpson's Lefschetz theorem, Delzant in \cite{Delzant10} gives a complete description of (the complement of) the BNS invariant for K\"ahler groups.

\begin{theorem}( \cite[Theorem 1.1]{Delzant10})\label{thm:delzant_bns_kahler}
Let $X$ be a compact K\"ahler manifold with the fundamental group $G$. For a non-trivial $\chi\in H^1(G;\R)$, we have $\chi$ is exceptional, i.e. $[\chi]\notin \Sigma(G)$,  if and only if there exists a holomorphic fibration (i.e. a holomorphic map with connected fibers) onto a hyperbolic Riemann orbifold $S$ of genus greater than or equal to $1$, say $f: X\to S$, such that $\chi\in f^*(H^1(\pi_1^{\mathrm{orb}}(S);\R))$.
\end{theorem}

The following conclusion is immediate, which is also stated in \cite[section 11.3]{PyBook}.
\begin{corollary}
(Symmetric property of K\"ahler groups)\label{coro:symmetric_bns_kahler} Any K\"ahler group $G$ has symmetric BNS invariant, i.e. 
$$\Sigma(G)=-\Sigma(G)$$
\end{corollary}

\begin{proof}
For any $[\chi]\notin \Sigma(G)$, by \cref{thm:delzant_bns_kahler}, we have holomorphic fibration $f: X\to S$ with $\chi=f^*(v)$ for some $v\in H^1(\pi_1^{\mathrm{orb}}(S);\R)$. Then $-\chi=f^*(-v)$ and $-v\in H^1(\pi_1^{\mathrm{orb}}(S);\R)$. Thus $[-\chi]\notin \Sigma(G)$, and the conclusion is proved.
\end{proof}

\section{The applications}\label{sec:hnn}

In this section, we give some applications to the study of K\"ahler groups.

The first result is the combination of \cite[Proposition 3.1 and Lemma 3.2]{FriedlVidussi16}. Their original proof uses the rank gradient, and we here prove the same result using solely the property of the BNS invariant.

\begin{proposition}\label{prop:amenable_kahler}
Let $G$ be a K\"ahler group which does not contain any free nonabelian subgroup, then $\Sigma(G)=\mathrm{S}(G)$ and the commutator subgroup $G'$ is finitely generated. In particular, if $G$ is an amenable K\"ahler group, we have the above conclusion holds.
\end{proposition}
\begin{proof}
By \cref{prop:property_BNS_union}, we have $\Sigma(G)\cup -\Sigma(G)=\mathrm{S}(G)$. Also, $\Sigma(G)=-\Sigma(G)$ by \cref{coro:symmetric_bns_kahler}, then we have the desired equality. The property of $G'$ follows directly from \cref{prop:bns_abelian_quotient}.
\end{proof}

\begin{remark*}
The above conclusion for amenable K\"ahler groups can also be seen as follows. We may assume $b_1(G)>0$; otherwise, it is trivially true. If the $\Sigma^c(G)\coloneqq\mathrm{S}(G)-\Sigma(G)$ is nonempty, by \cref{thm:delzant_bns_kahler}, we have a holomorphic fibration onto a hyperbolic Riemann orbifold, say $f: X\to S$, which induces a surjection $f_*$ on fundamental groups. Since $\pi_1(S)$ is not amenable while $\pi_1(X)$ is amenable, this contradicts the fact that the homomorphic image of an amenable group is still amenable. This argument was mentioned to me by Pierre Py.
\end{remark*}

Next, we will use the HNN-valuation introduced in \cite{Brown87} to reprove some classical results on K\"ahler groups in \cite{NR2008}. Notice a conflict in the ascending (resp. descending) HNN decomposition definition in the above two papers, and we will use the same notation as in \cite{NR2008} here. These results are also achieved by Fridel and Vidussi in \cite{FriedlVidussi16} using the rank gradient.

\begin{definition}[HNN decomposition]\label{HNN-decomposition}
Take any base group $B$ and a stable letter $t$. Let $\varphi: B_1\to B_2$ be an isomorphism between two subgroups $B_1, B_2$ of $B$. We say a group $G$ \textit{admits an HNN decomposition with base group} $B$ if it can be written as
$$G=\langle B,t\ |\ h^{t}\coloneqq t h t^{-1}=\varphi(h),\forall h\in B_1\rangle$$
If $B_1=B$ (and $B_2\neq B$), we say that $G$ is (properly) descending; if $B_2=B$ (and $B_1\neq B$), we say that $G$ is (properly) ascending.
\end{definition}

For a group $G$ admitting an HNN decomposition with base group $B$ and stable letter $t$, we can define a surjective group homomorphism $\chi: G\twoheadrightarrow  \Z$ by setting 
\begin{equation}\label{eq: associated chi}
\chi(g)=\begin{cases}
    1 & \text{\ if\ }g=t\\
    0 & \text{\ if\ }g\in B
\end{cases}
\end{equation}
We will call this $\chi$ the \textit{associated homomorphism.} Set $N=\ker(\chi)$, and obviously we have $N=\bigcup_{k\in\Z}t^k B t^{-k}$. 

Now we need the concept of an HNN valuation introduced by Brown, which can be seen as a generalization of the HNN decompositions (see \cite[Proposition 2.1(v)]{Brown87} and the remark afterwards).

\begin{definition}[HNN valuation]
Let $G$ be a group, and $\chi: G\to (\R,+)$ be a given group homomorphism. A function $v: G\to \R_\infty\coloneqq\R\cup\{\infty\}$ is called \textit{an HNN valution with respect to $\chi$} if it satisfies:
\begin{enumerate}[label=(\alph*)]
    \item $v(g^{-1})=v(g)+\chi(g)$.
    \item $v(gh)\ge \min\{v(g),v(h)-\chi(g)\}$
\end{enumerate}
We will say that $v$ is \textit{non-trivial} if in addition:\\
(c) $v|_{G_{\chi\le 0}}$ does not assume a minimal value.
\end{definition}
If $\chi\ne 0$, the condition (c) is equivalent to $v|_{N}$ is not bounded below, where $N=\ker(\chi)$ (\cite[Proposition 2.3]{Brown87}).

Note here $\chi: G\to \R$ can be any given group homomorphism. When $G$ is a group admitting an HNN decomposition, we have a natural choice of $\chi: G\to \Z\hookrightarrow\R$, say the associated homomorphism defined by \cref{eq: associated chi}.

The main criterion we will apply is:

\begin{proposition}(\cite[Proposition 3.1]{Brown87})\label{prop:HNN-valuation}
Let $G$ be a finitely generated group and $\chi:G\twoheadrightarrow\Z$ a surjective group homomorphism. The following three conditions are equivalent:
\begin{enumerate}[label=(\theenumi)]
\item $G$ admits a descending HNN decomposition with finitely generated base group $B$ and with $\chi$ as the associated homomorphism. 
\item $G$ does not admit a properly ascending HNN decomposition with $\chi$ as the associated homomorphism.
\item There is no non-trivial HNN valuation $v$ on $G$ with respect to this $\chi$.
\end{enumerate}
Moreover, if these conditions hold, then every HNN decomposition of $G$ with $\chi$ as associated homomorphism is descending.
\end{proposition}

With this notation at hand, Brown in \cite{Brown87} defined his invariant to be the set of classes $[\chi]\in \mathrm{S}(G)\coloneqq(\text{Hom}(G;\R)-\{0\})/ \R^{+}$ such that one representation (equivalently, all representations) $\chi$ satisfies condition (3) above with $\Z$ replaced by $\R$, i.e., there is no non-trivial HNN valuation $v$ on $G$ with respect to this $\chi$. He later, in \cite[Theorem 5.2]{Brown87}, showed that his invariant coincides with the BNS invariant in \cite{BNS87}. In conclusion, we have 
\begin{equation}\label{eq: bns by brown}
\Sigma(G)
=\left\{\, [\chi]\in \mathrm{S}(G)\ \middle|\ 
\begin{aligned}
 &\text{there is no non-trivial HNN valuation }\\
 &v \text{ on } G \text{\ with respect to\ } \chi
\end{aligned}
\,\right\}.
\end{equation}

We can now reprove the theorems in \cite{NR2008} on HNN decompositions.

\begin{theorem}(\cite[Theorem 0.2(b)]{NR2008})\label{thm: nr 1}
Let $G$ be the fundamental group of a compact K\"ahler manifold $X$. If $G$ admits a properly ascending HNN decomposition, then $X$ admits a holomorphic fibration onto a hyperbolic Riemann orbifold.
\end{theorem}
\begin{proof}
Let $\chi: G\to \Z$ be the associated homomorphism for the properly ascending HNN decomposition. Since condition (2) of \cref{prop:HNN-valuation} is not satisfied, neither is condition (3) there. Then there exists a non-trivial HNN valuation $v$ on $G$ with respect to this $\chi$, which means $[\chi] \notin \Sigma(G)$ by \cref{eq: bns by brown}. Then $X$ fibers over a hyperbolic Riemann orbifold according to \cref{thm:delzant_bns_kahler}. 
\end{proof}

\begin{theorem}(\cite[Theorem 0.3 (b)]{NR2008})\label{thm: nr 2}
Any group $G$ admitting a properly ascending (or descending) HNN decomposition with a finitely generated base group is not a K\"ahler group.
\end{theorem}
\begin{proof}
We first consider the case $G$ admits a properly ascending HNN decomposition with a finitely generated base group $B$ and stable letter $t$, and write
$$G=\langle B,t|t h t^{-1}=\varphi(h)\in B, \forall h\in B_1\rangle$$
with $B_1\subsetneq B$ and $\varphi: B_1\to B$ is a group isomorphism. By assumption on $G$, the condition (2) of \cref{prop:HNN-valuation} is not satisfied, and neither is (3) for the associated homomorphism $\chi$ with $\chi(t)=1$. Thus, we have $[\chi]\notin \Sigma(G)$ by \cref{eq: bns by brown}.

On the other hand, taking the stable letter $s=t^{-1}$ instead, we can write
$$G=\langle B,s|shs^{-1}=\varphi^{-1}(h)\in B_1, \forall h\in B\rangle$$
which is a (properly) descending HNN decomposition of $G$. Note the associated homomorphism now is $-\chi$, since we need its value at $s=t^{-1}$ to be 1 as in \cref{eq: associated chi}. Now we need the assumption that $B$ is finitely generated, which ensures condition (1) in \cref{prop:HNN-valuation} is satisfied with $-\chi$ as the associated homomorphism. Then (3) is also satisfied for $-\chi$, and we have $[-\chi]\in \Sigma(G)$ by \cref{eq: bns by brown}. This proves that $G$ can not be a K\"ahler group, since it contradicts the symmetric property of K\"ahler groups (\cref{coro:symmetric_bns_kahler}).

If $G$ admits a properly descending HNN decomposition with a finitely generated base group $B$ and $\chi$ is the associated homomorphism, we can similarly prove $[\chi]\in \Sigma(G)$ and $[-\chi]\notin \Sigma(G)$, and thus $G$ is not K\"ahler.
\end{proof}

\section*{Acknowledgments}
This work is inspired by the new book ``Lectures on K\"ahler groups'' by Professor Pierre Py, and the author would like to thank him for this aspect. The author also thanks Professor Yongqiang Liu for many useful conversations and the referee for helpful comments that make Section 3 especially more self-contained and clearer. The author is partially supported by the China Postdoctoral Science Foundation (reference 2023M744396) and China Scholarship Council (No. 202406340174).

\bibliographystyle{ssmfalpha} % We choose the "plain" reference style
\bibliography{main}

\end{document}